\documentclass[10pt, article]{amsart}

\usepackage{tikz}
\usetikzlibrary{calc}

\usepackage{ae} 
\usepackage[T1]{fontenc}
\usepackage[cp1250]{inputenc}
\usepackage{amsmath}
\usepackage{amssymb, amsfonts,amscd,verbatim}
\usepackage{hyperref}

\usepackage{indentfirst}
\usepackage{latexsym}
\input xy
\xyoption{all}

\usepackage{amsmath}    

\theoremstyle{plain}
\newtheorem{proposition}{Proposition}[section]
\newtheorem{theorem}[proposition]{Theorem}
\newtheorem{corollary}[proposition]{Corollary}
\newtheorem{conjecture}[proposition]{Conjecture}

\newtheorem{lemma}[proposition]{Lemma}

\theoremstyle{definition}
\newtheorem{definition}[proposition]{Definition}

\theoremstyle{remark}
\newtheorem{remark}[proposition]{Remark}

\newtheorem{example}[proposition]{\bf Example}

\newcommand{\im}{\operatorname{Im}\nolimits}

\newcommand{\diam}{\operatorname{Diam}\nolimits}

\newcommand{\Rips}{\operatorname{Rips}\nolimits}

\def\g{\gamma}

\def\di{\partial}

\def\f{\varphi}

\def\RR{{\mathbb R}}

\def\ZZ{{\mathbb Z}}

\def\FF{{\mathbb F}}

\def\f{{\varphi}}
\def\B{\mathcal{B}}

\def\N{\mathcal{N}}
\def\M{\mathcal{M}}

\numberwithin{equation}{section}
\title[Contractions in persistence and metric graphs]
{Contractions in persistence and metric graphs}

\author{\v Ziga ~Virk}
\address{University of Ljubljana, Slovenia}
\email{ziga.virk@fri.uni-lj.si}
\thanks{The author would like to thank Teresa Heiss and Arseniy Akopyan for useful discussions on this subject. Research was  supported by Slovenian Research Agency grants No. N1-0114 and P1-0292.}

\begin{document}

\maketitle
\begin{center}
\today
\end{center}

\begin{abstract}
We prove that the existence of a $1$-Lipschitz retraction (a contraction) from a space $X$ onto its subspace $A$ implies the persistence diagram of $A$ embeds into the persistence diagram of $X$. As a tool we introduce tight injections of persistence modules as maps inducing the said embeddings. We show contractions always exist onto shortest loops in metric graphs and conjecture on existence of contractions in planar metric graphs onto all loops of a shortest homology basis. 

Of primary interest are contractions onto loops in geodesic spaces. These act as ideal circular coordinates. Furthermore, as the Theorem of Adamaszek and Adams describes the pattern of persistence diagram of $S^1$, a contraction $X \to S^1$ implies the same pattern appears in persistence diagram of $X$.
\end{abstract}

\section{Introduction}

Persistent homology is a parameterized version of homology that has received a lot of attention in the past two decades. The inherent nature of the scale parameter results in features of persistent homology that are absent in standard homology: persistent homology is stable and also encodes geometric information about the underlying space. Nowadays, the study of persistent homology encompasses a multitude of aspects including algorithmic, geometric, algebro-topological, stochastic, and data analytical points of view. However, despite all the progress little is known about the way persistent homology encodes geometry of the underlying space or how to interpret persistence diagrams, both of which are fundamental questions.

The purpose of this paper is to provide a new interpretation of parts of persistent homology generated via Rips filtrations in terms of the geometry of an underlying space. Given a compact metric space $X$ and a $1$-Lipschitz retraction (called a contraction) $X \to A$ onto a subspace $A\subset X$, we show how persistent homology of $A$ appears within the persistent homology of $X$ itself via the inclusion induced map, see Proposition \ref{PropScale2}. As the \textbf{main result} we prove that in such a case even the persistence diagram of $A$ appears as a subset of the persistence diagram of $X$ (Corollary \ref{CorMain1}). For the purpose of the latter statement we introduce tight embeddings of persistence modules and show they induce inclusions on persistence diagrams (Theorem \ref{ThmTightEmbedding}). This is a property a generic embedding of persistence modules does not posses. We conclude the paper by demonstrating the existence of contractions onto shortest loops in metric graphs (Theorem \ref{ThmContrMetGraphs}) and conjecture contractions on loops of shortest homology basis always exist on planar metric graphs (Conjecture \ref{Conj}).

The importance of our results stems from their interpretative capacity. Suppose we are given an elementary subspace, say a simple geodesically closed loop (i.e., a geodesic circle) $\alpha \subset X$ in a Riemannian manifold $X$. If there is a contraction $X \to \alpha$ our results show that persistence diagram of $\alpha$ is contained in persistence diagram of $X$, see Figure \ref{FigEssence} for an example. As persistence diagram of a geodesic $S^1$ is known \cite{AA} to consist of odd-dimensional points, the same points also appear in persistence diagram of $X$ and thus we are able to deduce parts of the latter. Going in the opposite direction, if we are given a persistence diagram of $X$ (potentially as an approximation, via stability result, arising from a computation of a sample of $X$) which contains the pattern of a persistence diagram of $S^1$, we might expect to find a geodesic circle within $X$. Analogous conclusions may be made for other subspaces of $X$ for which at least a part of persistent homology is known such as certain ellipses and regular polygons, see Related work below.

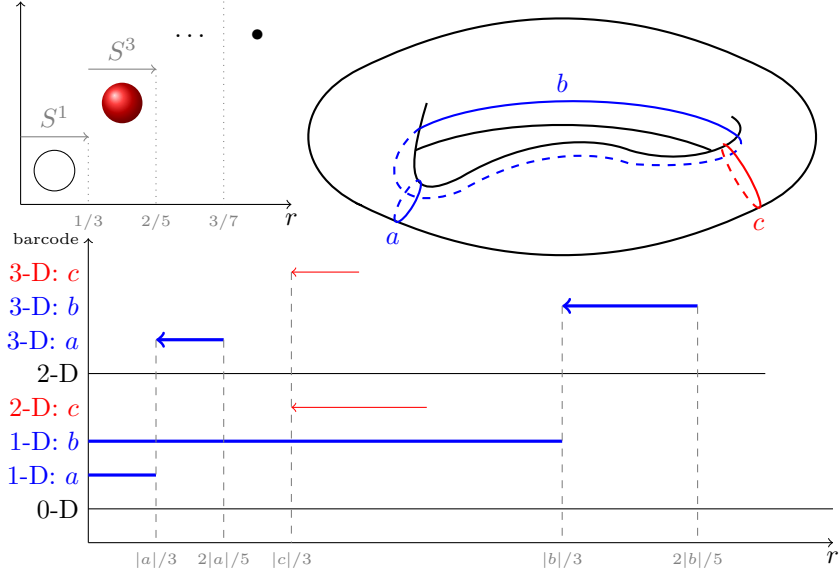
\begin{figure}
\begin{tikzpicture}[scale=.9]
\draw [thick] (-3,0) ..controls (-1,-1)and(1,-1)..
(3,0)..controls(4,.5)and(4, 1.5)..
(3,2)..controls(1,3)and(-1, 3)..
(-3,2)..controls(-4,1.5)and(-4, .5)..
cycle;
\draw [thick](-2,1.5) ..controls (-2.3,.5)and(-2.3, 0)..
(-1.5, .4)..controls(-.5,1) and(.5,1)..
(1,.8)..controls(2, 0.5)and(3, 1)..(2.5,1.3);
\draw  [thick](-2.17,.8) ..controls (-1,1.3)and(1,1.3)..(2.2,.8);
\draw [blue, thick, rotate=-30](-2,-1.45) arc (-90:90:.1 and .35);
\draw [blue, thick, dashed, rotate=-30](-2,-1.45) arc (-90:-270:.1 and .35);
\draw [blue](-2.5, -.5) node {$a$};
\draw [blue, thick](2.6,1) arc (20:150:2.6 and .8);
\draw [thick, blue, dashed](-2.1,1.13) ..controls (-2.8,.6)and(-2.5, -.2)..
(-1.5, 0.2)..controls(-.5,.8) and(.5,.8)..
(1,.6)..controls(2.3, 0.5)and(2.8, .8)..(2.6,1.0);
\draw [blue](0, 1.8) node {$b$};
\draw [red, thick, rotate=30](2.5,-1.5) arc (-90:90:.1 and .54);
\draw [red, thick, dashed, rotate=30](2.5,-1.5) arc (-90:-270:.1 and .55);
\draw [red](2.9, -.3) node {$c$};
\draw[->] (-8, 0) to (-4, 0) node[below]{$r$};
\draw[->] (-8, 0) to (-8, 3) node{};
\draw[->] [gray] (-8,1) to node [above]{$S^1$}(-7, 1);
\draw[dotted, gray]  (-7,1) to (-7, -0) node[below]{\tiny{$1/3$}};
\draw[->] [gray] (-7,2) to node [above]{$S^3$}(-6, 2);
\draw[dotted, gray]  (-6,2) to (-6, 0) node[below]{\tiny{$2/5$}};
\draw (-5.5, 2.5) node {$\ldots$};
\draw[dotted, gray]  (-5,3) to (-5, 0) node[below]{\tiny{$3/7$}};
\draw (-4.5, 2.5) node {$\bullet$};
\draw  (-7.5,.5) circle [radius=0.3];
\shade [ball color=red]  (-6.5,1.5) circle (0.3);
\draw [->] (-7, -5) to (4, -5) node[below] {$r$};
\draw [->] (-7, -5) to (-7, -.5) node[ left] {\tiny barcode};
\draw (-7, -4.5) node[left]{$0$-D} to (4, -4.5);
\draw [blue] (-7, -4)node[left]{ $1$-D: \color{blue} $a$} ;
\draw [blue, very thick] (-7, -4)  to (-6, -4);


\draw [blue] (-7, -3.5)node[left]{ $1$-D: \color{blue} $b$};
\draw [blue, very thick] (-7, -3.5) to (0, -3.5);


\draw (-7, -3)node[left]{\color{red} $2$-D:  $c$};
\draw [<-,red] (-4, -3) to (-2, -3);

\draw  (-7, -2.5)node[left]{ $2$-D} to (3, -2.5) node [right] {};

\draw  [blue](-7, -2)node[left]{ $3$-D: \color{blue} $a$};
\draw  [<-,blue, very thick](-6, -2) to (-5, -2) node [right] {};

\draw (-7, -1.5)node[left]{\color{blue} $3$-D:  $b$};
\draw [<-, blue, very thick] (0, -1.5) to (2, -1.5) node [right] {};
\draw (-7, -1)node[left]{\color{red} $3$-D:  $c$};
\draw [<-, red] (-4, -1)to (-3, -1) node [right] {};



\draw[dashed, gray](-6, -2) to (-6, -5) node[below]{\tiny$|a|/3$};
\draw[dashed, gray](-5, -2) to (-5, -5) node[below]{\tiny$2|a|/5$};

\draw[dashed, gray](-4, -1) to (-4, -5) node[below]{\tiny$|c|/3$};
\draw[dashed, gray](0, -1.5) to (0, -5) node[below]{\tiny$|b|/3$};
\draw[dashed, gray](2, -1.5) to (2, -5) node[below]{\tiny$2|b|/5$};

\end{tikzpicture}
\caption{The upper left side represents the homotopy type of the Rips filtration of a circle (i.e., odd-dimensional spheres) equipped with a geodesic metric by \cite{AA}. The black shape on the right is a two-dimensional torus $X$ and below it is an excerpt from its barcode in dimensions up to $3$. 
There exist contractions $X \to a$ and $X \to b$, yielding the thick odd-dimensional bars corresponding to these two loops by Corollary \ref{CorMain1}. The homology of a torus appears at $r=0$ by \cite{Haus, ZV3} while the later 2- and 3-dimensional bars arising from a geodesic loop $c$ appear by \cite{ZV2}. Note that there is no contraction $X \to c$. The results of \cite{ZV2} require a wide enough neighborhood of a geodesic loop and cannot be used in the case of $b$.
}
\label{FigEssence}
\end{figure}

Besides the mentioned interpretative capacity, our results raise new questions and connections. The first of them is an existence of contractions: it would be of interest to know when such maps $X \to A$ exist, especially onto a simple loop $A=\alpha$. Proposition \ref{PropContr1} gives a simple required condition in geodesic spaces: $\alpha$ should be a member of a shortest homology basis. Example \ref{Ex01} demonstrates that this condition is not sufficient and leads to Conjecture \ref{Conj}. Contractions $X \to \alpha$ actually represent $1$-Lipschitz cohomology classes in dimension $1$ (inspiring Example \ref{Ex02}) and would represent ideal circular coordinates in the sense of \cite{Circ}. A related question would be to determine optimal Lipschitz constants of cohomology classes as maps to Eilenberg-Maclane spaces. On the other hand, existence of contractions could be rephrased as extension problems and in fact a particular case of the Kirszbraun theorem \cite{Kir34,WeWi75}: find all $A \subseteq X$ for which the identity on $A$
extends to a $1$-Lipschitz map on $X$.

The treatment of this paper is taylored for (open) Rips filtrations and induced persistent homology and homotopy groups. Analogous arguments could be made for closed Rips filtrations, any Cech filtration, and more general settings of persistence modules.

----------------
\subsection{Related work}
The following are known results about the geometric information encoded in persistent homology. 

At small scales the Rips complexes of tame spaces attain the homotopy type of the underlying space \cite{Haus, Lat, ZV3}. The entire homotopy type of a Rips filtration is essentially only known in one non-trivial case: $S^1$ \cite{AA}. The methods of \cite{AA} can be used to extract some further results on ellipses \cite{Ad5} and regular polygons \cite{Ad4}. The entire $1$-dimensional persistent homology (and fundamental group) of geodesic spaces has been completely classified in \cite{ZV, ZV1}. Paper \cite{ZV2} (and also \cite{ZV4}) contains a local version of the result of this paper: if a subset $A \subset X$ has a sufficiently nice neighborhood, then parts of its persistent homology embed into persistent homology of $X$. The technical assumptions of these results hold for loops $a$ and $c$ of Figure \ref{FigEssence}, but not $b$. The assumptions of our main results of this paper are much easier to verify and in some cases hold more generally. Overall, persistent homology in dimensions $1,2$, and $3$ is known to encode some geodesic circles and shortest $1$-homology basis by \cite{ZV, ZV2, ZV4} (and now also by results of this paper), properties of thick-thin decomposition \cite{ACos} and injectivity radius \cite{Memoli}. On a similar note, the systole of a geodesic space is detected as the first critical scale of persistent fundamental group \cite{ZV}. Parts of persistent homology of certain spheres have been detected via stability theorem yielding a counterexample to the Hausmann's conjecture \cite{ZVCounterex}.

\section{Preliminaries}
\label{SecPrelims}

We briefly recall the notions used throughout the paper. For extended background see \cite{EH} for persistent homology, \cite{ZV} for persistent fundamental group of geodesic spaces, and \cite{ChaObs} for persistence diagrams and barcodes.

Given a metric space $X$ and $x\in X$, the open ball around $x$ of radius $r>0$ is denoted by $B(x,r),$ while the closed $r$-ball is denoted by $\overline B(x,r)$.
A map $f\colon X \to Y$ between metric spaces is $L$-\textbf{Lipschitz} for $L>0$ if 
$$
d_Y(f(x),f(y)) \leq L \cdot d_X(x,y), \quad \forall x,y\in X.
$$ 
A metric space $X$ is \textbf{geodesic}, if for each $x,y \in X$ there is an isometric embedding $[0,d(x,y)]\to X$ mapping $0 \mapsto x$ and $d(x,y)\mapsto y$.
Given a closed subspace $A \subset X$ of a topological space $X$, a \textbf{retraction} $X$ to $A$ is a map $f\colon X \to A$ satisfying $f|_{A}=id_A.$ 


Given a metric space $X$ and $r>0$, the \textbf{Rips complex} $\Rips(X,r)$ is an abstract simplicial complex with the vertex set being $X$, and a finite $\sigma \subset X$ being a simplex iff $\diam(\sigma) < r$.


Given a metric space $X$ and an interval $J \subseteq \RR$,  the Rips \textbf{filtration} over $J$ is a collection $\{\Rips(X,r)\}_{r \in J}$ of simplicial complexes along with the simplicial inclusion maps $$\rho_{r,r'}\colon \Rips(C,r) \hookrightarrow \Rips(X,r),$$ which are identities on vertices for each $r \leq r'$. When $J = (0,\infty)$ we refer to the filtration simply as the Rips filtration. 
 Given an Abelian group $G$, $n\in \{0,1,\ldots\}$ and a basepoint $\bullet \in X$ we apply the homology $H_n(\_ ; G)$ or homotopy group $\pi_n(\_ ,\bullet)$ functor to a filtration to obtain \textbf{persistent homology groups} $\{H_n(\Rips(X,r);G)\}_{r \in J}$ and \textbf{persistent homotopy groups} $\{\pi_n(\Rips(X,r),\bullet)\}_{r \in J}$. Each of these is also equipped with induced (and consequently commuting) homomophisms. These are denoted by $\rho^{G,n}_{r,r'}\colon H_n(\Rips(X,r);G) \hookrightarrow H_n(\Rips(X,r');G)$ for persistent homology and 
 $\rho^{\pi_n}_{r,r'}\colon \pi_n(\Rips(X,r),\bullet) \hookrightarrow \pi_n(\Rips(X,r'),\bullet)$ for persistent homotopy groups. 
 
 Given a  field $\FF$ and an interval $J \subseteq \RR$, a \textbf{persistence module} $\M$ over $J$ is a collection of vector spaces $\{M_r\}_{r \in J}$  
 and commuting linear bonding maps $\rho_{r,r'}: M_r \to M_{r'}$.
 Given a  field $\FF$ and an interval $J' \subset J \subseteq \RR$, the \textbf{interval module} $\FF_{J'}$ is a collection of vector spaces $\{V_r\}_{r \in J}$ with 
\begin{itemize}
 \item $V_r = \FF$ for $r\in J'$;
 \item $V_r=0$ for $r \notin J'$,
\end{itemize}
 and commuting linear bonding maps $V_r \to V_{r'}$ which are identities whenever possible (i.e., for $r,r'\in J'$) and zero elsewhere. Each persistent homology (with coefficients in $\FF$) of a Rips filtration over $J$ built upon a compact metric space is a persistence module that decomposes (uniquely up to permutation of the summands) as a direct sum of interval modules (see q-tameness condition in Proposition 5.1 of \cite{Cha2}, the property of being radical in \cite{ChaObs}, and the main result in \cite{ChaObs} along with its corollaries for details). The underlying intervals of the said collection of interval modules are called \textbf{bars} and form a multiset called \textbf{barcode} of the persistence module. For each bar, its endpoints form a pair of numbers from $(0,\infty)\cup \{\infty\}$. These pairs form a multiset called a \textbf{persistence diagram}. For each element of a barcode or a persistence diagram, its \textbf{multiplicity} is the number of repetitions of the element in the said multiset. The persistence diagram of $n$-dimensional homology with coefficients in $\FF$ of a compact metric space $X$ built via open Rips complexes on an open interval $J \subseteq \RR$ is denoted by $PD(\{H_n(\Rips(X,r);\FF)\}_{r\in J})$, while the corresponding barcode is $\B(\{H_n(\Rips(X,r);\FF)\}_{r\in J})$. A barcode also encodes the nature of the endpoints of its bars and hence contains more information than a persistence diagram. However, in our setting the nature of the endpoints is ``fixed'', see Lemma \ref{LemIntervals}, and hence both structure contain the same information.
 

\section{Tight inclusions of persistence modules}
\label{SecPMs}

Fix an open interval $(a,b) \subseteq \RR$ and a field $\FF$. Given persistence modules $\M=\{M_r\}_{r \in J}$ with bonding maps $\rho^M_{r,r'}$ and $\N=\{N_r\}_{r \in J}$ with bonding maps $\rho^N_{r,r'}$, an \textbf{inclusion} $\f$ of $\M$ into $\N$ is a collection of injective linear maps $\f_r \colon M_r \to N_r$ commuting with the bonding maps $\rho^*_{*,*}$. Inclusions of persistence modules do not induce inclusions of barcodes or persistence diagrams. For example, over $\RR$, $\FF_{[2,3]}$ can be included into $\FF_{[1,3]}$ yet the persistence diagrams are disjoint. Inclusions of persistence diagrams have been shown to only prolong the ``embedding bars'' to the left (and not to the right) in case of pointwise finite-dimensional persistence modules \cite{Bauer}. Contractions, on the other hand, will be shown to induce embeddings on persistence diagrams. Working towards the proof of this statement we introduce a particular kind of inclusions of persistence modules that induce inclusions on persistence diagrams.

\begin{definition}
Fix an open interval $J \subseteq \RR$ and a field $\FF$. Inclusion $\f =\{\f_r\}_{r\in J}$ of persistence module $\M=\{M_r\}_{r \in J}$ into persistence module $\N=\{N_r\}_{r \in J}$ is \textbf{tight}, if $\forall r',r \in J, r' < r$ we have
$$
\im \rho^M_{r',r} = M_r \cap \im \rho^N_{r',r}.
$$
\end{definition}

Informally speaking, tight inclusions do not bring back the emergence of homology classes to an earlier scale, but rather include their emergence ``tightly''.

The following lemma describes the types of bars emerging from Rips filtrations. While restricting to such bars in Theorem \ref{ThmTightEmbedding} is not strictly necessary, it will simplify our treatment.

\begin{lemma}
 \label{LemIntervals}
Assume $X$ is a compact metric space, $n\in \{0,1,\ldots\}$, and $\FF$ is a field. Then each bar of $\B(\{H_n(\Rips(X,r);\FF)\}_{r\in \RR})$ is of form $(a,b]$ or $(a,\infty)$ for some $0 \leq a < b <\infty$.
\end{lemma}

\begin{proof}
 Take a cycle $\alpha$ from $\Rips(X,r)$ representing a bar. Due to the strict inequality appearing in the definition of Rips complexes, there exists $r' < r$ such that $\alpha$ is also a cycle in $\Rips(X,r')$. Hence the bar is open at $a$. The same argument for a nullhomology implies that if $\alpha$ is nullhomologous in some $\Rips(X,r)$, it is also nullhomologous in some $\Rips(X,r')$ for some $r' < r$. 
\end{proof}

\begin{corollary}
 \label{CorIntervals}
Assume $X$ is a compact metric space, $n\in \{0,1,\ldots\}$,  $\FF$ is a field, and $J = (j_1,j_2)\subset \RR$ is an open interval. Then each bar of $\B(\{H_n(\Rips(X,r);\FF)\}_{r\in J})$ is of form $(a,b]$ or $(a,j_2)$ for some $j_1 \leq a < b <j_2$.
\end{corollary}

\begin{proof}
 We can decompose $\B(\{H_n(\Rips(X,r);\FF)\}_{r\in \RR})$ into interval modules. Restricting these interval modules to scale span $r\in J$ we obtain  a decomposition of $\B(\{H_n(\Rips(X,r);\FF)\}_{r\in J})$. The proof now follows from Lemma \ref{LemIntervals}.
\end{proof}

For the sake of clarity of the argument of Theorem \ref{ThmTightEmbedding} we state the following simple algebraic lemma. It can be proved straight from the definitions.

\begin{lemma}
 \label{LemAlgy1}
 Let $\FF$ be a field and suppose $W^N \leq V^N$ are finite dimensional vector spaces over $\FF$. Let $q\colon V^N \to V^N / W^N$ denote the natural quotient map. Then for each subspace $V^M \leq V^N$ we have:
\begin{itemize}
 \item $\ker (q|_{V^M}) = W_N \cap V^M$, and
 \item $\dim  q(V^M) = \dim V^M  - \dim (W^N \cap V^N)$.
\end{itemize}
\end{lemma}

The following theorem is the main result of this section. It states that tight inclusions of persistence modules induce inclusions of barcodes and persistence diagrams. Its formulation is tailored to our setting although it holds more generally.

\begin{theorem}
 \label{ThmTightEmbedding}
 Fix an open interval $J = (j_1,j_2)\subset \RR$ and a field $\FF$. Assume inclusion $\f =\{\f_r\}_{r\in J}$ of persistence module $\M=\{M_r\}_{r \in J}$ into persistence module $\N=\{N_r\}_{r \in J}$ is tight. If both persistence modules arise as persistent homology of compact metric spaces via open Rips filtrations (and hence admit the interval decompositions), then $\B(\M) \subseteq \B(\N)$ and $PD(\M) \subseteq PD(\N)$.
\end{theorem}

\begin{proof}
Throughout the proof we consider $M_r$ to be a subspace of $N_r$ via the inclusion $\f, \forall r.$
 By Corollary \ref{CorIntervals} we only have to consider two types of bars. First let us assume $I=(a,b] \subset J$ is a bar we consider. 
 
 Choose $t\in (a,b)$. For each $*\in \{\M, \N\}$ define (see \cite{Boe} or \cite{ChaObs} for background) $W^*_t \leq V^*_t \leq N_t$  as follows:
 $$
 V^*_t = \bigcap_{s\in (a,t)} \im \rho^*_{s,t} \cap \bigcap_{s>b} \ker \rho^*_{t,s},
 $$
 $$
 W^*_t = \left( \im \rho^*_{a,t} \cap \bigcap_{s>b} \ker \rho^*_{t,s} \right) + \left(  \bigcap_{s\in (a,t)} \im \rho^*_{s,t} \cap \ker \rho^*_{t,b} \right).
 $$
These expressions provide the number of bars of the form $(a',b']$ containing $t$ as follows:
\begin{itemize}
 \item $\dim V^*_t$ is the number of bars with $a' \leq a$ and $b' \leq b$;
  \item the dimension of the first term of $W^*_t$ is the number of bars with $a' < a$ and $b' \leq b$;
    \item the dimension of the second term of $W^*_t$ is the number of bars with $a' \leq a$ and $b' < b$.
\end{itemize}
The multiplicity of bar $(a,b]$ in a barcode of $*\in \{\M, \N\}$ is $\mu^* = \dim (V^*_t / W^*_t) = \dim V^*_t - \dim W^*_t$. Within this setup we state two claims:
\begin{description}
 \item[Claim 1] $V^M_t \leq V^N_t$. This claim follows from our assumptions.
  \item[Claim 2] $W^N_t \cap V^M_t = W^M_t$. Let us prove this claim. As $V^M_t$ and $W^N_t$ both contain $W^M_t$ we have $W^N_t \cap V^M_t \supseteq W^M_t$. In order to prove the other inclusion choose $v\in V^M_t \cap W^N_t$:
	\begin{itemize}
 		\item for each $t'$  $v\in \ker \rho^N_{t,t'}$ implies $v\in \ker \rho^M_{t,t'}$;
		\item for each $t'$ $v\in \im \rho^N_{t',t}$ implies $v\in \ker \rho^M_{t',t}$ by the tight inclusion assumption as $v\in V^M_t \subseteq M_t$.
	\end{itemize}
	Combining these two implications with the conditions that $v\in W^N_t$ implies $v\in W^M_t$, which proves $W^N_t \cap V^M_t \subseteq W^M_t$ and thus Claim 2. 
\end{description}
Now let $\mu^M = \dim V^M_t/W^M_t$ and  $\mu^N= \dim V^N_t/W^N_t$ denote the multiplicity of $I$ in $\M$  and $\N$ respectively. Using the above claims and Lemma \ref{LemAlgy1} for the quotient map $q\colon V^N_t \to V^N_t \to V^N_t/W^N_t$ we conclude
$$
\mu^N = \dim V^N_t/W^N_t = \dim q(V^N_t) \geq  \dim q(V^M_t) = \dim V^M_t - \dim (W^N_t \cap V^N_t) =
$$
$$
= \dim V^M_t - \dim W^M_t = \mu^M.
$$
Hence $\f$ induces an injection on bars of form $I=(a,b]$.

Intervals of the form $(a,j_2)$ are treated in the same way by choosing $t \in (a,j_2)$ and defining
$$
 V^*_t = \bigcap_{s\in (a,t)} \im \rho^*_{s,t},
 $$
 $$
 W^*_t = \left( \im \rho^*_{a,t} \right) + \left(  \bigcap_{s\in (a,t)} \im \rho^*_{s,t} \cap \bigcap_{s>t} \ker \rho^*_{t,s} \right).
 $$
\end{proof}

\section{Contractions in persistence}
\label{SecPers}

Throughout this section let $A\subset X$ be a closed subspace of a metric space $X$ and let $i \colon A \hookrightarrow X$ be the associated inclusion. 

\begin{definition}
 Let $r>0$. An $r$-\textbf{contraction} of $X$ to $A$ is a retraction $f\colon X \to A$ for which
 $$
 d(x,y) < r \implies d(f(x),f(y)) < r.
 $$
 A maps $X \to A$ is a \textbf{contraction} if it is an $r$-contraction for each $r>0$.  
\end{definition}

\begin{remark}
 Contractions are $1$-Lipschitz retractions. It should be apparent that the property of being a contraction is much stronger than the property of being an $r$-contraction for some $r$.
\end{remark}

\begin{proposition} [Contractions induce retractions at single scale]
\label{PropScale1}
 Suppose $f \colon X \to A$ is an $r$-contraction for some $r>0$. Then:
\begin{enumerate}
 	\item The induced map $\bar f \colon \Rips(X,r) \to \Rips(A,r)$ is a simplicial retraction.
	\item Map $i$ induces injection on all homology and homotopy groups.
\end{enumerate}
\end{proposition}
 
\begin{proof}
 Part (1) follows straight from the definition. In order to prove (2) choose $n \in \{0,1,\ldots\}$ and a homology element in $H_n(\Rips(A,r);G)$ represented by an $n$-cycle $\alpha$. If $\alpha = \di \beta$ for some $(n+1)$-cycle in $\Rips(X,r)$, then $\bar f(\beta)$ (obtained by applying $f$ to vertices involved in $\beta$) is an $(n+1)$-cycle in $\Rips(A,r)$ demonstrating that $[\alpha]=0$, hence the statement holds for homology groups. The proof for homotopy groups is the same using simplicial representatives of maps.
\end{proof}
 
\begin{remark}
 Statement (1) of Proposition \ref{PropScale1} implies that $\bar f \circ \bar i = id|_{\Rips(A,r)}$. This is a particular case of homotopy dominance. 
\end{remark}
 
\begin{remark}
Contractions induce retractions on Rips complexes at all scales. In an analogous way, \cite{Haus} introduced crushings as maps which behave like deformation retraction on Rips complexes. Crushings were used under the name deformation contractions in \cite{ZV2} to prove local variant of the main result of this paper. 

In a similar manner, $r$-contractions induce retractions on Rips complexes at scale $r$. Analogous maps are $r$-crushings of \cite{Lat}, which induce deformation retractions on Rips complexes at scale $r$. 
\end{remark}

\begin{proposition} [Contractions induce retractions at multiple scales]
\label{PropScale2}
 Let $J \subseteq \RR$ be an open interval. Suppose $f \colon X \to A$ is an $r$-contraction for all $r \in J$. Then:
\begin{enumerate}
	\item Map $i$ induces injection on all persistent homology and persistent homotopy groups on the interval $r \in J$.
	\item For any $\FF$ and $n\in \{0,1,\ldots\}$ the inclusion 
	$$
	\{H_n(\Rips(A,r);\FF)\}_{r\in J} \hookrightarrow \{H_n(\Rips(X,r);\FF)\}_{r\in J}
	$$
	of persistence modules induced by $i$ is tight.
\end{enumerate}
\end{proposition}

\begin{proof}
 Statement (1) is apparent from Proposition \ref{PropScale1} and definitions. In order to prove statement (2) choose an $n$-cycle $\alpha$ representing an element in $H_n(\Rips(A,r);\FF)$. If $[\alpha] =\rho^{\FF,n}_{r',r}[\beta]$ for some  $n$-cycle $\beta$ representing an element in $H_n(\Rips(X,r');\FF)$, then (as in Proposition \ref{PropScale1}), $[\alpha] =\rho^{\FF,n}_{r',r}[\bar f (\beta)]$ with $\bar f (\beta)$ being an $n$-cycle in \newline $H_n(\Rips(A,r');\FF)$. Indeed: if $\alpha - \beta = \di \gamma$ for some $(n+1)$-chain in $\Rips(X,r)$, then 
 $$
\di \bar f (\gamma)=\bar f (\di \gamma)=  \bar f (\alpha - \beta) = \alpha -  \bar f ( \beta).
 $$
\end{proof}

\begin{theorem}
\label{ThmMain1}
 Let $J \subseteq \RR$ be an open interval, $X$ a compact metric space, $\FF$ a field, and $n\in \{0,1,\ldots\}$. Suppose $f \colon X \to A$ is an $r$-contraction for all $r \in J$. Then there are inclusions of barcodes and persistence diagrams:
\begin{itemize}
 \item $\B(\{H_n(\Rips(A,r);\FF)\}_{r\in J})\subseteq \B(\{H_n(\Rips(X,r);\FF)\}_{r\in J})$ and
 \item $PD(\{H_n(\Rips(A,r);\FF)\}_{r\in J})\subseteq PD(\{H_n(\Rips(X,r);\FF)\}_{r\in J})$.
\end{itemize}
\end{theorem}

\begin{proof}
 By (2) of Proposition \ref{PropScale2} the inclusion of persistence modules exists and is tight. By Theorem \ref{ThmTightEmbedding} the interval decompositions exist and the tight inclusion induces inclusions of barcodes and persistence diagrams.
\end{proof}


\begin{corollary}
 \label{CorMain1}
 Let  $X$ be a compact metric space, $\FF$ a field, and $n\in \{0,1,\ldots\}$. Suppose $f \colon X \to A$ is a contraction. Then there are inclusions of barcodes and persistence diagrams:
\begin{itemize}
 \item $\B(\{H_n(\Rips(A,r);\FF)\}_{r\in \RR})\subseteq \B(\{H_n(\Rips(X,r);\FF)\}_{r\in \RR})$ and
 \item $PD(\{H_n(\Rips(A,r);\FF)\}_{r\in \RR})\subseteq PD(\{H_n(\Rips(X,r);\FF)\}_{r\in \RR})$.
\end{itemize}
\end{corollary}

With these results we are able to justify interpretation of the example provided in Figure \ref{FigEssence}. 

\section{Contractions in metric graphs}
\label{SecMG}

In this section we discuss existence of contractions on metric graphs. We are particularly interested in contractions onto loops isometric to a geodesic $S^1$ as these are essentially the only spaces for which we know the entire persistent homology. 

\begin{definition}
 Given a geodesic space $X$, a \textbf{geodesic circle} is a simple closed loop in $X$, whose subspace metric makes it a geodesic space. 
\end{definition}

Geodesic circles have been discussed in our context in \cite{ZV, ZV2}. 

\begin{proposition}
\label{PropContr1}
 Suppose $X$ is a compact, locally contractible geodesic space. If there exists a contraction $f\colon X \to \alpha$ onto a simple closed curve $\alpha \subset X$, then $\alpha$ is a geodesic circle and $[\alpha]$ is a member of a lexicographically shortest homology basis of $H_1(X;G)$ in any coefficients $G$. 
\end{proposition}

\begin{proof}
Let $\ell$ be the length of $\alpha$ and choose an Abelian group $G$.
Assume $[\alpha]=[\beta] + [\gamma]$, with $\beta$ and $\gamma$ being loops in $X$ of lengths shorter than $\ell$. Let $\tilde f \colon H_1(X; G)\to H_1(\alpha;G)$ be the map induced by $f$. Then,  computing in $H_1(\alpha;G)$, we have
$$
1 =[\alpha]=[f(\alpha)]= \tilde f [\alpha]=\tilde f [\beta] + \tilde f [\gamma]= [f(\beta)] + [f(\gamma)]. 
$$
Hence at least one of the last two terms, say $[f(\beta)]$, is non-trivial. This means that the winding number of $f(\beta)$ in $\alpha$ is non-trivial and thus $f(\beta)$ is of length at least $\ell$. But this is a contradiction as a contraction decreases the length and $\beta$ was assumed to be of length less than $\ell$. Hence $\alpha$ is a member of a lexicographically shortest homology base of $H_1(X;G)$.

 By \cite{ZV} a compact locally contractible space has a finite lexicographically shortest homology basis of $H_1(X;G)$ in any coefficients $G$, and all members of any such basis are geodesic circles. 
\end{proof}

A natural follow-up question is whether the required condition of Proposition \ref{PropContr1} for the existence of a contraction onto a simple loop is sufficient. We will answer that question in a negative way in the context of metric graphs by Example \ref{Ex01}. 

\begin{definition}
 A \textbf{metric graph} is a geodesic space homeomorphic to a finite $1$-dimensional simplicial complex.
\end{definition}

The following example was suggested by Arseniy Akopyan. 

\begin{example}
\label{Ex01}
Consider two concentric circles of lengths
    $1000$ and $999$, along with additional connections as 
    shown by Figure \ref{fig:counterexample}. The inner circle $A$ of length $999$ induces a member of a shortest homology basis [A] in homology. However, there is no contraction of this metric 
    graph onto $A$, as any contraction would 
    have to map a loop of length $1993$
    that goes around most of the inner loop once and around 
    most of the outer loop once,
    changing between them at the cross at the bottom, twice around $A$.
 \begin{figure}
  \includegraphics{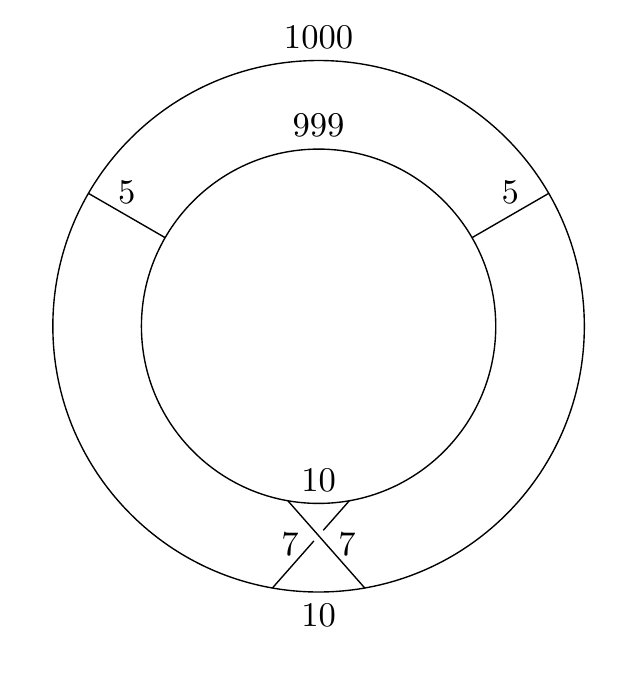}
  \caption{A sketch of Example \ref{Ex01}. Numbers $1000$ and $999$ indicate the lengths of concentric circles, while other numbers indicate the lengths of short segments. 
}
  \label{fig:counterexample}
\end{figure}
\end{example}

\begin{example}
\label{Ex02}
A similar example disproving the converse of Proposition \ref{PropContr1} can be designed by treating a standard flat Klein bottle $K$ obtained by identifying sides of a square. Recall that $H_1(K;\ZZ) \cong \ZZ \oplus \ZZ_2$, meaning that the shortest homology basis has at two elements (corresponding to ``horizontal'' and ``vertical'' lines on the defining square). However,  due to the shift in torsion we have $H^1(K;\ZZ) \cong \ZZ$ and as the elements of the later cohomology groups correspond to homotopy classes of maps to $K(\ZZ,1)=S^1$, there isn't even a continuous retraction of $K$ onto a loop generating torsion in $H_1(K;\ZZ)$.
\end{example}

In both these examples we appear to have used the fact that the space is not planar. This motivates the following conjecture. 

\begin{conjecture}
\label{Conj}
Suppose $X$ is a planar metric graph and $\alpha \subset X$ is a simple closed loop such that $[\alpha]$ is a member of a shortest homology basis of $H_1(X;G)$. We conjecture that then there exists a contraction $X \to \alpha$.
\end{conjecture}

A positive answer to a conjecture could be combined with the main results of this paper and \cite{AA} to show that for each member of a shortest homology basis of $H_1(X;G)$ of a planar metric graph $X$ the barcodes of $\{H_*(\Rips(X,r);\FF)\}_{r > 0}$ contain corresponding odd-dimensional bars in all odd dimensions (similarly to Corollary \ref{CorLast}).

We end this section by proving that a contraction onto a shortest loop in a metric graph always exists. Throughout the forthcoming proof we will be using the following simple fact: given an injective path $\gamma$ in a metric graph and another path $\gamma'$ with the same endpoints, if the interiors of $\gamma$ and $\g'$ are disjoint, then the paths form a non-contractible loop.

\begin{theorem}
\label{ThmContrMetGraphs}
 Suppose $X$ is a metric graph and $\alpha \subset X$ is a shortest (non-contractible) loop in $X$. Then there exists a contraction $X \to \alpha$.
\end{theorem}

\begin{proof}
By Proposition \ref{PropContr1} $\alpha$ is a geodesic circle. Let $2\ell$ be the length of $\alpha$.
Fix a point $a\in \alpha$ and let $a'\in \alpha$ be the point opposite to $a$, i.e., $d(a,a')=\ell$. We may assume that neither $a$ nor $a'$ is a vertex of $X$. Define 
$$
U = \bigcup_{x\in \alpha} B(x,d(a,x)).
$$
Furthermore, for each $p\in \alpha$ let $T_p$ denote the set of all points of $U$, whose closest point on $\alpha$ is $p$. Observe that $U = \bigcup_{p\in \alpha} T_p$. Let $\mathcal{P} \subset \alpha$ denote the (finite) subset of all points $p\in \alpha$, for which $T_p$ is not a singleton. Observe that $a' \notin \mathcal{P}$. We proceed by two claims.

\textbf{Claim 1}: $T_p \cap T_q = \emptyset, \forall p \neq q.$ If this was not the case there would exist $p\neq q\in \alpha$, and $v\in T_p \cap T_q$. We could then choose geodesics $\gamma_p$ from $v$ to $p$ and $\gamma_q$ from $v$ to $q$. Let $\gamma'$ be a geodesic from $p$ to $q$ along $\alpha$. Concatenating $\g_p, \g_q,$ and $\g'$ we obtain a loop $\g$. As the interior of $\g'$ is disjoint from $\g_p$ and $\g_q$, the loop $\g$ is not contractible. Its length is
$$
d(v,p)+ d(v,q) + d(p,q) < d(a,p) + d(a,q) + d(p,q) \leq 2 \ell,
$$ 
which contradicts the assumption of $\alpha$ being a shortest non-contractible loop. This proves Claim 1.

\textbf{Claim 2}: For each $p\in \mathcal{P}, T_p$ is a tree. Working towards the proof of the claim we again assume the conclusion does not hold, i.e., we assume there exists a simple closed loop $\beta$ in $T_p$ for some $p\in \alpha$. Let $b\in \beta$ be a point closest to $p$ and let $\gamma_b$ be a geodesic between the two points. As the length of $\beta$ is larger than $2 \ell$, we can choose a point $c\in \beta$ such that the length of $\gamma'$, which is defined as a shortest segment along $\beta$ from $b$ to $c$, equals $d(p,a)-d(p,b)$. Let $\gamma_c$ denote a geodesic from $c$ to $p$. As $c\in T_p$, its length is less than $d(a,p)$. Paths $\g_b, g', $ and $\g_c$ form a loop $\g$ in $T_p$. Paths $\g_b$ and $\g'$ only intersect at $a$ by definition and their concatenation is of length $d(p,a)$. Path $\g_c$ is shorter and thus $\g$ is not contractible. The length of $\g$ is 
$$
d(p,b) + \big(d(p,a)-d(p,b)\big) + d(c,p) < d(p,a) + d(p,a) \leq 2\ell.
$$
Again, this is a contradiction with our assumptions and thus Claim 2 holds.

We proceed by defining a contraction $f \colon X \to \alpha$. The map can informally be described as ``combing $U$ towards $a$ along $\alpha$'', see Figure \ref{Fig1}. In particular, for $x \in T_p$ define $f(x)$ as the point on the geodesic segment from $p$ to $a$ along $\alpha$ with $d(p,f(x))=d(p,x)$. By the claims above this defines a continuous map on each $T_p$ and also on $U$. For $x \notin U$ we define $f(x)=a$. We next show $f$ is a contraction. Let $\g$ be a geodesic between $x,y\in X$. We can decompose $\g$ into segments such that each segment is contained in a single edge of $X$ and also either in $U$ or $X\setminus U$. By definition $f$ maps each segment either via an isometric embedding to $\alpha$ or to a constant map at $a$. Hence the length of $f(\g)$ does not exceed the length of $\g$. As $X$ is geodesic this implies $f$ is a contraction (and in particular, continuous). 
\end{proof}

\begin{figure}[htbp]
\begin{center}
\includegraphics{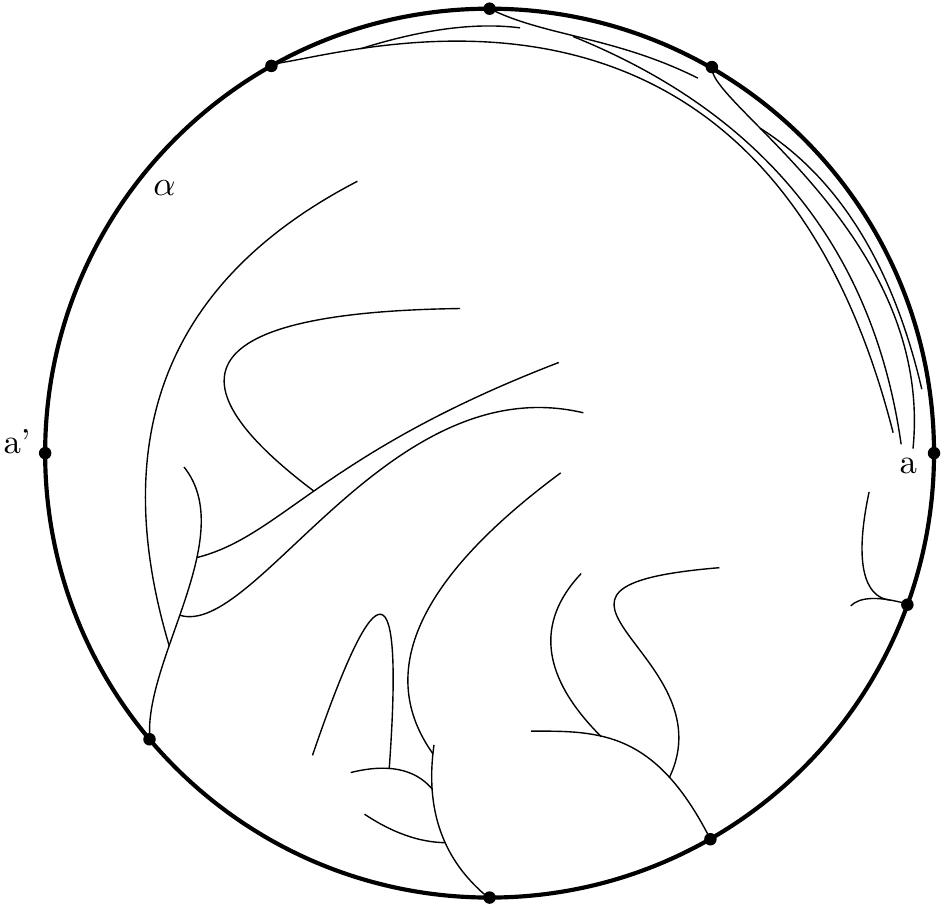}
\caption{A sketch of set $U$ from the proof of Theorem \ref{ThmContrMetGraphs}: trees $T_p$ are attached to loop $\alpha$. Attached trees on the upper half are combed towards $a$ indicating function $f$.}
\label{Fig1}
\end{center}
\end{figure}

\begin{corollary}
\label{CorLast}
 Suppose $X$ is a metric graph and $\alpha \subset X$ is a shortest loop in $X$. Let $\ell$ be the length of $\alpha$. Then for each $k \in \{0,1,2,\ldots\}$, the barcode $$\B(\{H_{2k+1}(\Rips(X,r);\FF)\}_{r > 0})$$ contains a bar $\left( \frac{k \ell}{2k + 1}, \frac{(k+1)\ell}{2k + 3} \right]$ induced by inclusion $\alpha \hookrightarrow X$.
\end{corollary}

\begin{proof}
 The statement holds by Theorem \ref{ThmContrMetGraphs} and Corollary \ref{CorMain1}.
\end{proof}


\end{document}